\documentclass[10pt]{article}
\font\smallit=cmti10

\usepackage{amssymb,latexsym,amsmath,epsfig,amsthm} 
\usepackage{mathtools}
\usepackage{hyperref}

\usepackage{amsrefs}
\BibSpec{article}{%
 +{} {\PrintAuthors}{author}
 +{, } {}{title}
 +{, } {  \textit}{journal}
 +{ } { \textbf}{volume}
 +{ }{ \parenthesize}{date}
 +{, }{ }{pages}
 +{} { \parenthesize}{language}
 +{.}{}{transition}
}

\makeatletter

\renewcommand\section{\@startsection {section}{1}{\z@}
{-30pt \@plus -1ex \@minus -.2ex}
{2.3ex \@plus.2ex}
{\normalfont\normalsize\bfseries\boldmath}}

\renewcommand\subsection{\@startsection{subsection}{2}{\z@}
{-3.25ex\@plus -1ex \@minus -.2ex}
{1.5ex \@plus .2ex}
{\normalfont\normalsize\bfseries\boldmath}}

\renewcommand{\@seccntformat}[1]{\csname the#1\endcsname. }

\makeatother

\newtheorem{theorem}{Theorem}
\newtheorem{lemma}{Lemma}

\theoremstyle{remark}
\newtheorem{question}{Question} 
\newtheorem*{theoremunnumbered}{Theorem}

\newcommand{\NN}{\ensuremath{\mathbb N}}
\newcommand{\ZZ}{\ensuremath{\mathbb Z}}
\newcommand{\QQ}{\ensuremath{\mathbb Q}}
\newcommand{\RR}{\ensuremath{\mathbb R}}

\DeclareMathOperator{\Lim}{Lim}
\DeclareMathOperator{\ot}{ot}

\begin{document}
\begin{center}
\uppercase{\bf Representing Ordinal Numbers with Arithmetically Interesting
  Sets of Real Numbers}
\vskip 20pt
{\bf D. Dakota Blair}\\
{\smallit Department of Mathematics, The Graduate Center at the City University\\
 of New York, 365 Fifth Avenue New York, NY 10016 USA}\\
{\tt dakota@tensen.net}\\ 
\vskip 10pt
{\bf Joel David Hamkins}\\
{\smallit Professor of Logic, Oxford University \&\ Sir Peter Strawson Fellow,\\ University College, Oxford, High Street Oxford OX1 4BH U.K.}\\
{\tt joeldavid.hamkins@philosophy.ox.ac.uk}\\ 
\vskip 10pt
{\bf Kevin O'Bryant\footnote{Support for this project was provided by a PSC-CUNY Award, jointly funded by The Professional Staff Congress and The City University of New York.}}\\
{\smallit Department of Mathematics, College of Staten
  Island \&\ The Graduate Center at The City University of New York, Staten Island, NY 10314, USA}\\
{\tt obryant@gmail.com}\\ 
\end{center}


\centerline{\bf Abstract}
\noindent
For a real number $x$ and set of natural numbers $A$, define 
	\(x \ast A \coloneqq \{ x a \bmod 1: a\in A\}\subseteq [0,1).\)
We consider relationships between $x$, $A$, and the order-type of $x\ast A$. For example, for every irrational $x$ and countable order-type $\alpha$, there is an $A$ with $x\ast A \simeq \alpha$, but if $\alpha$ is a well order, then $A$ must be a thin set. If, however, $A$ is restricted to be a subset of the powers of 2, then not every order type is possible, although arbitrarily large countable well orders arise.

\pagestyle{myheadings} 
\thispagestyle{empty} 
\baselineskip=12.875pt 
\vskip 30pt


\section{Introduction}
For any real number $x$ and $A\subseteq \NN$, the set
	\[
	x \ast A \coloneqq  \{ x a \bmod 1: a\in A\} \subseteq [0,1)			
	\]
has long held interest for number theorists. Principally, the distribution of the sequence $(x a_i \bmod 1)_{i\in\NN}$ in the interval $[0,1)$ has impacted areas as diverse as the study of exponential sums and numerical integration.

In the present work, we consider the order type of the set $x \ast
A$. Technically, a \emph{well order} is an ordered set in which every nonempty set has a least element, and an \emph{ordinal} is the order type of a well-order. In this work, we use the terms interchangeably. We will make free use Cantor's
notation for ordinals. The reader may enjoy John Baez's lighthearted online introduction~\cites{Baez1,Baez2,Baez3}, or the more traditional~\cite{MR662564}.

First, we address a few trivialities. If $x$ is rational with denominator $q$, then 
	\(x\ast A\subseteq \left\{0,\frac 1q, \frac 2q, \ldots,\frac{q-1}q\right\},\) 
and so $x \ast A  \preceq q$ (when comparing ordinals, we use the
customary $\succ,\succeq,\prec, \preceq, \simeq$). Also, if $A$ is
finite, then $x\ast A \preceq |A|$. Conversely, if $x$ is irrational
and $A$ is infinite, then $x\ast A$ is infinite and countable.

The general problem we consider is which irrationals $x$, infinite
sets $A\subseteq \NN$, and countable order-types $\alpha$ have the relation
  \[ x\ast A \simeq \alpha.\]

The easiest examples, as often happens, arise from Fibonacci numbers. Let $\phi$ be the golden ratio and
	\[\scr{F} \coloneqq  \{F_2,F_3,F_4,\ldots\} = \{1,2,3,5, 8, 13, \ldots\}.\]
It is well-known that $|\phi F_n - F_{n+1}|\to 0$ monotonically, with
$\phi F_n - F_{n+1}$ alternating signs. Therefore, $\phi \ast \scr{F}$
has two limit points, 0 and 1, and consequently has the same order
type as $\ZZ$. Taking the positive even indexed Fibonacci numbers
	\[\scr{F}_{\text{even}} = \{F_{2i} : i \in \NN, i\ge 1 \}\]
and shifting by 1 yields some other small ordinals: for $k\geq 0$
\[
  \phi^{2k+2} \ast  (\scr{F}_{\text{even}}+1) \simeq \omega ,\qquad
  \phi^{2k+1} \ast (\scr{F}_{\text{even}}+1) \simeq \omega + 2\cdot k
\]
The observation that inspired us to undertake this study is that the ordinal property is preserved by taking sumsets, and in particular
        \[ x\ast h \scr{F}_{\text{even}} \simeq \omega^h.\]

Following each theorem statement, we indicate a related question we haven't been able to answer. Our first general result is that we can always ``solve'' for $A$, in a very strong sense.  
\begin{theorem}\label{thm:every order type arises}
Let $\alpha_0,\alpha_1,\ldots,\alpha_{k-1}$ be any countable order types, and let $x_0,\ldots,x_{k-1}$ be any irrational numbers with $1,x_0,x_1,\ldots,x_{k-1}$ linearly independent over $\QQ$. There is a set $A\subseteq\NN$ such that for $i\in\{0,1,\ldots,k-1\}$,
	\[ x_i \ast A \simeq \alpha_i .\]
The set $A$ can be taken arbitrarily thin, in the sense that for any $\Psi:\NN \to \NN$ tending to $\infty$, we can take $A$ to have 
	\( | A \cap [0,n) | \leq \Psi(n)\)
for all $n\in\NN$.
        
If every $\alpha$ is an ordinal, then $A$ must have density 0, but for any $\Theta:\NN \to \NN$ with $\Theta(n)/n \to 0$, we can take $A$ to have infinitely many $n\in \NN$ with
	\( |A \cap [0,n) | > \Theta(n).\)
\end{theorem}

\begin{question}
  Is there a stronger way to say ``$A$ cannot be arbitrarily thick''? For example, it seems plausible that it is always possible to choose $A$ so that there is a positive constant $C$ with $|A \cap [0,n)| \geq C \log n$ for all $n$, while it seems implausible that we can always take $A$ so that $|A \cap [0,n)| \geq C\frac n{\log n}$.

  The condition that \emph{every} $\alpha$ is an ordinal, in the last paragraph of Theorem~\ref{thm:every order type arises}, is strictly stronger than is needed. Unfortunately, we have not found a nice way to express the actual requirement.
\end{question}

\begin{theorem} \label{thm:sumsets}
Let $x$ be an irrational number, and let $0\leq a_0<a_1< \cdots$ be a sequence of integers with $a_i x \bmod 1$ increasing to 1 monotonically. Let $A=\{a_0,a_1,\ldots\}$, and for any positive integer $h$ let $hA$ be the $h$-fold sumset of $A$. Then
	\[x \ast hA \simeq \omega^h.\]
\end{theorem}

\begin{question}
If $x\ast A$ is an ordinal, then $x\ast hA$ must be an ordinal, too. Can one give bounds on that ordinal?
\end{question}

\begin{theorem} \label{thm:powers}
  Fix  $b\geq 2$, and set $B=\{b^i : i\in \NN\}$. For any countable ordinal $\alpha$, there is an $x$ with $x\ast B \succeq \alpha$. There is no $x$ with $x\ast B$ an ordinal and $\omega \preceq x\ast B \prec \omega^2$.
\end{theorem}

\begin{question}
  Are there other voids, or can every countable ordinal at least as large as $\omega^2$ be represented? For example, can $\omega^2+1$ be represented with powers of 2?
\end{question}


\section{Proofs}

\begin{proof}[Proof of Theorem~\ref{thm:every order type arises}.]
We use a theorem of Weyl~\cite{Weyl}.
\begin{theoremunnumbered}[Weyl's Equidistribution Theorem]
If $1,x_0,x_1,\ldots,x_{k-1}$ are linearly independent over $\QQ$, then for any intervals $I^{(i)}$, $(0\leq i < k)$ with lengths $\lambda(I^{(i)})$,
	\[\lim_{N\to\infty} \frac{\left|\{n : 0\leq n < N, nx_i \in I^{(i)},  0 \leq i < k\} \right|}{N} = \prod_{0\leq i <k} \lambda(I^{(i)}).\]
\end{theoremunnumbered}

Suppose that $I_0, I_1, \dots$ is a sequence of disjoint nonempty intervals. Then this sequence has an order type, where we say that interval $I_i$ is less than interval $I_j$ if every element of $I_i$ is less than every element of $I_j$.\footnote{Suppose that $a_0,a_1,\ldots$ are comparable objects. Define $f(a_i)$ by $f(a_0)=(1/3,2/3)$, and $f(a_i)$ to be any interval in $(0,1)$ that has a positive distance from each of  $f(a_0),f(a_1),\ldots,f(a_{i-1})$ and is in the correct gap that so that $a_i,a_j$ have the same order as $f(a_i),f(a_j)$, for all $0\leq j <i$. Then the intervals $\{f(a_i) : i\in \NN\}$ have the same order type as the $a_0,a_1,\ldots$.}  Since the rational line is universal for countable orders, we can realize each of the countable order types $\alpha_i$ as the order type of a sequence $I_0^{(i)},I_1^{(i)}, \ldots$ of disjoint nonempty open intervals. By Weyl's theorem, which requires our irrationality condition on the $x_i$, for each $k$-tuple of natural numbers $\vec m = \langle m_0,m_1,\ldots,m_{k-1}\rangle$, the set
\begin{equation}\label{eq:define n_m}
  \left\{ n \in \NN : n x_i \bmod 1 \in I_{m_i}^{(i)}, 0 \leq i <k \right\}
\end{equation}
is infinite, and we set $n_{\vec m}$ to be any element of it. In particular, for each $i$, each of the intervals $I_0^{(i)},I_1^{(i)}, \ldots$ contains exactly one point of the form $n_{\vec m} x_i$ (as $\vec m$ varies). 
This means that the set
	\[
	A\coloneqq  \left\{ n_{\vec m} : \vec m \in \NN^k \right\}
	\]
has the needed property: $x_i \ast A \simeq \alpha_i$. Since the sets in \eqref{eq:define n_m} are infinite, we can choose $n_{\vec m}$ so as to make $|A \cap [0,n)| \leq \Psi(n)$.

Now, assume that $x\ast A$ is an ordinal (it is enough to show for $k=1$). We need to show that the density of $A$ is 0.  Since $x\ast A$ is an ordinal, for each $z$ the set $\{y\in x\ast A : z<y\}$ is either empty or has a least element. That is, each $z\in x\ast A$ is either the maximal element of $x\ast A$ or else has a successor. Let $z_0,z_1,\ldots$ be an enumeration of $x\ast A$. If $z_i$ has a predecessor and a successor, then set $J_i=(z_i,z_i^+)$, where $z_i^+$ is the successor of $z_i$. If $z_i$ has a predecessor but no successor, then set $J_i=(z_i,1)$. If $z_i$ does not have a predecessor but does have a successor, then set 
	$$J_i=\bigg(\limsup \big(x\ast A\cap[0,z_i)\big), z_i\bigg) \cup (z_i, z_i^+).$$ 
If $z_i$ has neither a predecessor nor a successor, then set $J_i=(0,z_i)\cup(z_i,1)$.

We have partitioned $(0,1)$ into $(x\ast A)\setminus\{0\}$ and $J_0,J_1,\ldots$. The disjoint open intervals making up $J_0,J_1,\ldots$ cover almost all of [0,1), i.e., $\sum_{i=0}^\infty \lambda(J_i)=1$. The sets
	\[A_i \coloneqq  \{ k \in \NN : x k \bmod 1 \in J_i\}\]
are pairwise disjoint because the $J_i$ are, and $d(A_i) = \lambda(J_i)$ by Weyl's Theorem, and $A_i \cap A = \emptyset$ by construction. Thus, the complement of $A$ is $\cup_i A_i$, and $d(\cup_i A_i) = \sum_i \lambda(J_i) = 1$. The set $A$ must have density 0.

Assuming now that all of the $\alpha_i$ are ordinals, we show how to augment $A$ so as to have $|A\cap[0,n)| > \Theta(n)$ for infinitely many $n$ without changing the ordinals. Let $z_i$ be the smallest limit point of $x_i\ast A$, which must exist as $x_i\ast A$ is infinite, and must be strictly positive as $x_i\ast A$ does not have an infinite decreasing subsequence. Let $J_0^{(i)}\coloneqq (0,z_i)$. By Weyl's Theorem, the set
	\[ A_0 \coloneqq  \{ n \in \NN : 0 \leq n, \, nx_i \bmod 1 \in J_0^{(i)} ,0\leq i < k\}\]
has density $z\coloneqq z_0\cdots z_{k-1}$, which is positive, and so $|A_0 \cap [0,n)| > (z/2) n$ for all sufficiently large $n$. Choose $n_0$ so that $\Theta(n_0) < (z/2) n_0 < |A_0 \cap [0,n_0)|$, which is possible by the hypothesis that $\Theta(n)/n \to 0$. Let
\[A_0' = A_0 \cap [0,n_0],\]
so that $|A_0'| > (z/2) n_0$ and $x_i\ast A_0' \subseteq  J_0^{(i)}$.
Now for $m\geq 1$ set $J_m^{(i)}=(z_i-z_i/2^m,z_i)$, and
  \[A_m \coloneqq  \{ n \in \NN : n_{m-1} \leq n, \, nx_i \bmod 1 \in J_m^{(i)}, 0 \leq i < k\}, \]
a set with density $z/2^{mk}$. Choose $n_m>n_{m-1}$ so that
	\[ \Theta(n_m) < \frac{z}{2\cdot 2^{mk}} n_m < |A_m \cap [0,n_m) |. \]
Set $A_m' = A_m \cap [n_{m-1},n_m)$ Then the set $x_i \ast \cup_m A_m'$ has order type $\omega$ for each $i$, and counting function that exceeds $\Theta(n)$ at each of $n_0,n_1,\ldots$. We can replace $A$ with  $A\cup \bigcup_m A_m'$, and $A$ has the same order type as before, and now has a counting function guaranteed to beat $\Theta$.
\end{proof}

\begin{proof}[Sketch of Proof of Theorem~\ref{thm:sumsets}.]
  For $X$ a set of real numbers, let $\Lim(X)$ be the derived set of $X$, i.e., the set of limit points of $X$. If $X$ is an ordinal, then so is $\Lim(X)$. If $X$ is also infinite and bounded, then
  $$X \preceq \omega \cdot \Lim(X) \preceq X+\omega.$$ 
If $X\subseteq[0,1)$ and $1\in \Lim(X)$, then $X \simeq \omega \cdot
\Lim(X)$.

Theorem~\ref{thm:sumsets} is clearly true for $h=1$; assume henceforth that $h\geq2$. We first prove that $x\ast hA$ is contained in $[0,1)$, is an infinite ordinal, has 1 as a limit point, and that $\Lim(x\ast hA) =\{1\} \cup \bigcup_{r=1}^{h-1} x\ast rA$. From this we conclude by induction that $x\ast hA \simeq \omega \cdot \left( \bigcup_{r=1}^{h-1} x\ast rA \right) \simeq \omega^h$. 

By definition of ``$\ast$'', clearly $x\ast hA \subseteq[0,1)$. That
$x\ast hA$ is an infinite ordinal is a combination of the following
observations: $x\ast hA = h(x\ast A) \bmod 1$; if $X_i$ are ordinals,
then so is $\sum X_i$; if $X\subseteq \RR$ is bounded and an ordinal,
then so is $X\bmod 1$.

The elements of $hA$ have the form $a_{i^{(0)}}+a_{i^{(1)}}+\cdots+a_{i^{(h-1)}}$ with $i^{(0)} \leq i^{(1)} \leq \cdots \leq i^{(h-1)}$. Suppose that we have a sequence (indexed by $j$) in $x\ast hA$ that converges to $L$:
  \[ z_j \coloneqq  x \big( a_{i_j^{(0)}}+a_{i_j^{(1)}}+\cdots+a_{i_j^{(h-1)}} \big)  \bmod 1 \to L \in [0,1].\]
If $i_j^{(0)}\to\infty$ for this sequence, then each $a_{i_j^{(k)}}$ goes to infinity for $k\in\{0,\ldots,h-1\}$. As $x \ast a_i x \mod 1$ goes to $1$ from below, we know that $L=1$. Otherwise, we an pass to a subsequence on which $i_j^{(0)} =i^{(0)}$ is constant. Either $i_J^{(1)}$ is unbounded, in which case $L=a_{i^{(1)}}x\mod 1$, or we can pass to a subsequence on which $i_j^{(1)}=i^{(1)}$ is constant. Repeat for $i^{(2)}$, and so on, to get that the limit points are 1 and 
	\[ \bigcup_{r=1}^{h-1} x\ast rA. \]
\end{proof}

\section{When the multiplying set consists of powers of $b$}

Multiplying a real $x$ by a power of $b$ and reducing modulo 1 is just a shift of the base $b$ expansion of $x$. Consequently, in this section, we obtain some economy of thought and exposition if we consider the following equivalent\footnote{Not quite equivalent. The reals $0.0\overline{1}$, $0.1\overline{0}$ (in base 2)  are the same, while the words ${01^\omega}=0111\cdots $, $10^\omega=1000\cdots$ are not equal. However, since we only consider irrational reals (a property preserved by shifting), the non-uniqueness of $b$-ary expansions never arises.} formulation of the problem.

For (possibly infinite) words $W=w_0w_1w_2\cdots$ and $V=v_0v_1v_2\cdots$ with $w_k,v_k \in \NN$, we define $W<V$ if $i=\inf\{k\in\NN : w_k \neq v_k \}$ is defined and $w_i<v_i$. Moreover, we call $1/2^i$ the \emph{distance} between $W$ and $V$. Note that $W=01$ and $V=011$, for example, are incomparable in this ordering, as $w_2$ is not defined, much less satisfying $w_2<v_2$. We define the shift map $\sigma$ by $\sigma(w_0w_1w_2\cdots) = w_1w_2w_3\cdots$. If for all $k\in\NN$ one has $0 \le w_k<b$, we say that $W$ is a base-$b$ word. We define $\ot(W)$ to be the order type of the set of shifts of $W$,
\[ \ot(W) \simeq \{ \sigma^k(W) : k \in \NN \},\]
which are linearly ordered. We say that a word $W$ is irrational if it is infinite and there are no two distinct shifts $\sigma_1,\sigma_2$ with $\sigma_1(W) = \sigma_2(W)$. We use exponents as shorthand for repeated subwords, as in \mbox{$(3^501)^2=33333013333301$}. An exponent of $\omega$ indicates an infinite repetition.

An enlightening example shows that the next lemma is best possible. Let $w_i=0$ if $i$ is a triangular number\footnote{Triangular numbers (A000217) have the form $k(k+1)/2$. The first several are $0,1,3,6,10,15$.}, and $w_i=1$ otherwise. That is
  \[ W \coloneqq  0010110111011110\cdots =01^001^101^201^301^4\cdots =  \prod_{k=0}^{\infty} 01^k. \]
The limit points of shifts of $W$ are 
\begin{equation}\label{eq:lims}
  01^\omega,101^\omega, \ldots, 1^k01^\omega, \ldots.
  \end{equation}
As the words in~\eqref{eq:lims} have only one limit point, $1^\omega$, which is not a shift of $W$, and the limit points themselves have order type $\omega$, we find that $\ot(W) = \omega^2$.

\begin{lemma}\label{lem:void}
  Suppose that $X$ is an irrational word base-$b$ word, with $2\leq b <\infty$. Then $\ot(X)$ has infinitely many limit points. In particular, if $\ot(X)$ is an ordinal, then $\ot(X) \succeq \omega^2$.
\end{lemma}

This is striking, as Theorem~\ref{thm:every order type arises} states that every order-type can be represented as $x\ast A$ for any irrational $x$ and some $A$. This is a peculiar facet of the ``powers of $b$'' sets.

\begin{proof}
This is consequence of the proof that nonperiodic words have unbounded complexity. We include the proof here as it is a beautiful argument.

Let $S(n)$ be the set of those finite subwords of length $n$ that appear in $X$ infinitely many times, and let $C(n) \coloneqq  |S(n)|$. Clearly $C$ is nondecreasing and, as $X$ is irrational, $C(1)\geq 2$. 

Suppose, by way of contradiction, that $C$ is bounded, i.e., that there is an $n_1$ with $C(n) \leq n_1$ (for all $n$). As $C(1),C(2),C(3),\dots$ is a bounded nondecreasing sequence of natural numbers, there is some $m$ such that $C(m)=C(m+1)$. If $u\in S(m)$, then at least one of $u0,u1,\dots,u(b-1)$ is in $S(m+1)$. As $C(m)=C(m+1)$, however, we know that exactly \emph{one} of $u0,u1,\dots,u(b-1)$  is in $S(m+1)$. Therefore, exactly one of $\sigma(u0),\sigma(u1),\dots,\sigma(u(b-1))$ is in $S(m)$. The graph with vertex set $S(m)$ and a directed edge from each $u$ to whichever of $\sigma(u0),\sigma(u1),\dots,\sigma(u(b-1))$ is in $S(m)$ is finite, connected, and each vertex has out-degree 1. Therefore the graph is a cycle. Consequently, the word $X$ is eventually periodic. This is contradicts the assumption that $X$ is irrational, and so we conclude that $C$ is unbounded.

For each $u\in S(n)$, there are infinitely many shifts of $X$ in the interval
\( ( u0^\omega, u1^\omega) \)
and so by Bolzano-Weierstrauss those shifts have a limit point in the interval \( [ u0^\omega, u1^\omega] \). Therefore $\ot(X)$ has an unbounded number of limit points, which implies that $\ot(X) \succeq \omega^2$.
\end{proof}
  
\begin{lemma}\label{lem:base doesn't matter}
  Let $b,c$ be integers with $b\geq c \geq 2$.
  \begin{enumerate}\renewcommand{\labelenumi}{(\emph{\roman{enumi}})} 
  \item If $W$ is a base-$b$ word and $\ot(W)$ is an ordinal, then there is a base-$c$ word $V$ with $\ot(W)\preceq \ot(V)$, and $\ot(V)$ is an ordinal.\label{i}
  \item If $V$ is a base-$c$ word, then there is a base-$b$ word $W$ with $\ot(W) \simeq \ot(V)$.\label{ii}
  \end{enumerate} 
\end{lemma}

\begin{proof}
Part (\emph{ii}) is obvious, as a base-$c$ word \emph{is} a base-$b$ word. 
  
Let $W$ be a base-$b$ word with $\ot(W)$ an ordinal. Let $D$ be a word morphism\footnote{Word morphisms are defined on letters, but apply to words letter-by-letter.} defined by $D(d)=01^{d+1}$. For example,
  $$D(0130)=D(0)D(1)D(3)D(0) = (01)(011)(01111)(01) = 010110111101.$$
We note that $D(W)$ is a base-2 word.
  
First, we argue that $V\coloneqq D(W)$ is an ordinal. By way of contradiction, suppose that $v_0,v_1,\ldots$ is an infinite decreasing subsequence of shifts of $V$. In $V$, there are never consecutive 0's, and never $b+1$ consecutive 1's. Therefore, we can pass to an infinite subsequence of $(v_i)$ all of which start with $1^k0$ for some fixed $k$ with  $0\leq k \leq b$. If every word of the sequence starts with the same letter, we can shift that starting letter into oblivion without altering the decreasing property of the sequence. Therefore, without loss of generality, every one of the $v_i$ begins with a 0. But then, each $v_i$ is exactly the image (under $D$) of a shift of $W$, and as $\ot(W)$ is an ordinal, there is no infinite such sequence.

But clearly $x<y$ if and only if $D(v)<D(y)$, so that \mbox{$\ot(W) \preceq \ot(D(V))$}.
\end{proof}

\begin{lemma}\label{lem:big ordinals}
  Suppose that $w_0,w_1,\ldots$ is a sequence of words with $\ot(w_i)$ an ordinal for every $i\in\NN$. There is a base-$2$ word $W$ such that for every $i$, we have $\ot(w_i) \preceq \ot(W)$.
\end{lemma}

\begin{proof}
By Lemma~\ref{lem:base doesn't matter}, we can assume that the $w_i$ are base-2 words. Let $B_i(w)$ be the morphism (mapping base-2 words into base-3 words) that maps $0\mapsto  12^{i+1}$ and $1\mapsto 22^{i+1}$. In other words, $B$ sticks $i+1$ letter 1's between each pair of letters, and then replaces $0,1$ with $1,2$. Each $B_i(w_i)$ is an ordinal, and $\ot(w_i) \preceq \ot(B_i(w_i))$.  

Let $x_0=1,x_1=2,x_2,\ldots$ be the set of all finite subwords of all of the $B_i(w_i)$, organized first by length, and second by the order $<$ defined at the beginning of this section. Set
   $$V \coloneqq \prod_{k=0}^\infty 3^k x_k 0 = (x_00)(3x_10)(33x_20)\cdots.$$
We claim that $V$ is an ordinal, and that for each $i$, $\ot(B_i(w_i)) \preceq \ot(V)$.

Suppose, by way of contradiction, that $v_0,v_1,v_2,\dots$ is an infinite decreasing sequence of shifts of $V$. By passing to a subsequence, we may assume that each of $v_0,v_1,v_2,\ldots$ begins with the same letter. 

If they all begin with 3, the length of the initial string 3's must be nonincreasing (as $v_0,\dots$ is a decreasing sequence), and so by passing to a subsequence we may assume that each $v_i$ begins with a string of 3's of the same length.  If every one of a list of words begins with the same letter, applying the shift map does not change the ordering. In particular, we can apply the shift map to all of $v_0,v_1,v_2,\ldots$, and so we may assume that none of the $v_i$ begin with 3.

The shifts of $V$ that begin with 0 begin with $03^kx$ for some $k$ and some $x\in\{1,2\}$, and each $k$ only happens once. Therefore, there are no infinite decreasing sequences that all start with 0. By passing to a subsequence, we may assume that either all of $v_0,v_1,\ldots$ begin with 1, or all begin with a 2.

Assume, for the moment, that all begin with 2. In $V$, each string of 2's can be followed by either a 0 or a 1, and so the length of the initial string of 2's is nonincreasing. By passing to a subsequence, we can assume that the initial strings of 2's all have the same length.  By shifting, we come to an infinite decreasing subsequence of shifts of $V$, all of which begin with 0 or 1. By passing to a subsequence again, they all begin with 1.

That is, without loss of generality, all of $v_0,v_1,v_2,\dots$ begin with 1. The 1's all come from $x_i$'s, which come from some $B_j(w_j)$'s, but in $B_j(w_j)$ each 1 is followed by $2^{j+1}$. As the $v_0,v_1,v_2,\dots$ sequence is decreasing, the length of the string of 2's following the initial 1 is nonincreasing. By passing to a subsequence, we may assume that all of $v_0,v_1,v_2,\ldots$ begin with $12^{k+1}x$ for some nonnegative $k$ and some $x\in\{0,1\}$. But this means that each $v_i$ starts in a subword of $B_k(w_k)$.  As $B_k(w_k)$ is an ordinal, the position of the first 0 is bounded. By passing to a subsequence, that position is the same in every $v_0,v_1,v_2,\ldots$, and by shifting, each of $v_0,v_1,\dots$ begins with 0. But as noted above, there aren't infinite descending sequences in which every $v_i$ begins with 0.

Thus, $\ot(V)$ is an ordinal. By Lemma~\ref{lem:base doesn't matter}, there is a base-2 word $W$ with $\ot(V) \preceq \ot(W)$. That is, for each $i$
  \[\ot(w_i) \preceq \ot(B_i(w_i)) \preceq \ot(V) \preceq \ot(W).\]
\end{proof}

Theorem~\ref{thm:powers} follows immediately from Lemma~\ref{lem:big ordinals} and Lemma~\ref{lem:void}.

\end{document}